\documentclass[11pt]{article}

\usepackage{amsmath,amsfonts,amssymb,color}

\usepackage[left=2cm,top=2cm,right=3cm]{geometry}
\newcommand{\beq}{\begin{equation}}
\newcommand{\eeq}{\end{equation}}
\newcommand{\bea}{\begin{eqnarray}}
\newcommand{\eea}{\end{eqnarray}}
\newcommand{\beas}{\begin{eqnarray*}}
\newcommand{\eeas}{\end{eqnarray*}}

\newtheorem{theorem}{Theorem}[section]

\newtheorem{assumption}[theorem]{Assumption}
\newtheorem{definition}[theorem]{Definition}
\newtheorem{proposition}[theorem]{Proposition}

\newtheorem{corollary}[theorem]{Corollary}
\newtheorem{lemma}[theorem]{Lemma}
\newtheorem{remark}[theorem]{Remark}
\newtheorem{example}[theorem]{Example}
\newtheorem{examples}[theorem]{Examples}
\newtheorem{foo}[theorem]{Remarks}

\newenvironment{proof}{\addvspace{\medskipamount}\par\noindent{\it
Proof}.}
{\unskip\nobreak\hfill$\Box$\par\addvspace{\medskipamount}}

\newcommand{\R}{\mathbb R}

\newcommand{\ga}{\gamma}

\title{Smoothing effect of rough differential equations driven by fractional Brownian motions}
\author{Fabrice Baudoin\footnote{Department of Mathematics, Purdue University, 
West Lafayette, IN, USA}, Cheng Ouyang\footnote{Department of Math, Statistics and Computer Science, University of Illinois at Chicago, IL, USA}, Xuejing Zhang\footnote{Department of Mathematics, Purdue University, 
West Lafayette, IN, USA}}

\begin{document}

\maketitle

\begin{abstract}
In this work we study the smoothing effect of rough differential equations driven by a  fractional Brownian motion with parameter $H>1/4$. The regularization estimates we obtain generalize to the fractional Brownian motion previous results by Kusuoka and Stroock and can be seen as a quantitative version of the existence of smooth densities under H\"ormander's type conditions.
\end{abstract}

\baselineskip 0.25in

\tableofcontents

\newpage

\section{Introduction}

In this paper, we study stochastic differential equations driven by a fractional Brownian motion with Hurst parameter $H \in (1/4, 1)$.  More precisely, let us consider the equation
\begin{equation}
\label{eq:intro} X^{x}_t =x + \sum_{i=1}^d \int_0^t V_i (X^{x}_s) dB^i_s,
\end{equation}
where the vector fields $V_1,\ldots,V_d$ are $C^\infty$-bounded vector fields on $\R^n$ and where  $B$ is a $\mathbb{R}^d$-valued centered Gaussian process with covariance 
\[
\mathbb{E}(B_s \otimes B_t) =\frac{1}{2} \left( t^{2H}+s^{2H} -| t-s |^{2H}\right).
\]
The parameter $H$ is the so-called Hurst parameter of the fractional Brownian motion. It quantifies the sample path regularity of $B$ since a straightforward application of the Kolmogorov continuity theorem implies that the paths of $B$ are almost surely locally H\"older of index $H-\varepsilon$ for $0 < \varepsilon <H$. When $H=1/2$, $B$ is a Brownian motion. Fractional Brownian and equations driven by it appear as a natural model in biology and  physics (see for instance \cite{KS,SW, TBV}).

If $H>1/2$, then the paths of $B$ are regular enough and the equation \eqref{eq:intro} is understood in the sense of Young. Existence and uniqueness of solutions are well-known in that case (see \cite{NR,Za}). When $1/4<H\le 1/2$, it can be shown (see \cite{CQ}) that $B$ can canonically be lifted as a geometric $p$-rough path with $p >1/H$. As a consequence, rough paths theory (see \cite{FV-bk,LQ}) can be used to give a sense to what is solution of equation  \eqref{eq:intro}. In the case $H=1/2$, this notion of solution coincides with the solution of the corresponding Stratonovitch stochastic differential equation.

\

In the past few years, the study of the regularity of the law of $ X^{x}_t $ has generated great amount of work. In \cite{BH}, the authors prove, in the regular case $H>1/2$, that if the vector fields $V_1,\cdots,V_d$ satisfy the classical H\"ormander's bracket generating condition, then for $t >0$, the random variable $X_t^x$ admits a smooth density with respect to the Lebesgue measure. In \cite{CF}, the authors prove, in the case $H>1/4$, and under the same assumption on the vector fields, the existence of the density. The smoothness of this density is proved in \cite{HP} for $H>1/3$, conditioned on the integrality of the Jacobian of such systems which is established in \cite{CLL}.  Finally,  smoothness of the density function in the case $H>1/4$ is proved  in \cite{CHLT}.

\

The regularity of the law of $X_t^x$ is intimately related to the regularization properties of the operator:
\[
P_tf(x)=\mathbb{E}(f(X_t^x)),
\]
that is defined for a Borel and bounded function $f$. It should be denoted that when $H \neq 1/2$, $(P_t)_{t \ge 0}$ is not a semigroup and that there is no direct connection with the theory of partial differential equations unless the vector fields $V_1,\cdots,V_d$ commute (see \cite{BC} for further discussion on that aspect). By regularization property of $P_t$, we mean that $P_t$ has "smoothing" effect on the initial datum $f$: If $f$ is a Borel and bounded function $f$, then $P_t f$ is a smooth function for every $t>0$. In the Brownian motion case, that is if $H=1/2$, the regularization property of $P_t$ has been extensively studied and explicitly quantified by Kusuoka and Stroock \cite{KS1,KS2,KS3} and Kusuoka \cite{Ku}. In particular, in his work \cite{Ku}, Kusuoka introduces the  UFG condition on the vector fields (this is our Assumption  \ref{UH condition}) and proves that if this condition is satisfied, then the following theorem holds:

\begin{theorem}[Brownian motion case, Kusuoka \cite{Ku}]
 Let $x \in \mathbb{R}^n$. For any integer $k\ge 1$ and $0 \le i_1 , \cdots , i_k \le d$, there exists a constant $C>0$ (depending on $x$) such that for every $C^\infty$ bounded function $f$ and $t \in (0,1]$,
 \[
  \\|V_{i_1}\cdots V_{i_k }P_{t}f(x)\\|\le Ct^{-k/2} \| f \|_\infty.
 \]
\end{theorem}

The main purpose of the present paper is to generalize this statement to any $H\in (1/4,1)$. More precisely, we prove the following theorem:

\begin{theorem}[Fractional Brownian motion case]\label{main}
Assume $H \in (1/4,1)$ and that the vector fields $V_1,\cdots, V_d$ satisfy the Kusuoka's condition UFG (see Assumption  \ref{UH condition}).  Let $x \in \mathbb{R}^n$. For any integer $k\ge 1$ and $0 \le i_1 , \cdots , i_k \le d$, there exists a constant $C>0$ (depending on $x$) such that for every $C^\infty$ bounded function $f$ and $t \in (0,1]$,
 \[
  \\|V_{i_1}\cdots V_{i_k }P_{t}f(x)\\|\le Ct^{-Hk} \| f \|_\infty.
 \]
\end{theorem}

Our result is obviously an extension of Kusuoka's result, since it encompasses the case $H=1/2$. It is interesting to observe that the framework given by the most recent developments in  rough paths theory (see in particular \cite{ CHLT,CLL, HP})  actually simplifies Kusuoka's approach and, in our opinion,  provides an overall simpler and clearer proof of his result which originally built on \cite{KS1,KS2,KS3}.  

We should mention that Theorem \ref{main} was already proved by two of the authors in the regular case $H>1/2$ and under a strong ellipticity assumption on the vector fields, see \cite{BO12}. The rough setting and the more general UFG assumption on the vector fields make the proof of Theorem \ref{main} much more difficult.

The paper is organized as follows. In Section 2, we give the necessary background on Malliavin calculus that will be needed throughout the paper. Section 3 is devoted to the proof of  the main technical estimates that are needed. It is the heart of our contribution. In the Brownian motion case,  similar estimates are obtained in \cite{KS2,KS3,Ku}, but the proof of those heavily relies on  Markov and martingale methods. We prove here that such estimates may be obtained in a more general setting by using quantitative versions of Norris' type lemma (see \cite{BH} and \cite{HP}) which are based on interpolation inequalities and by using small ball probability estimates for fractional Brownian motions (see \cite{LS}). Once these estimates are obtained, after some work the integration by part technique of Kusuoka-Stroock \cite{KS3} and Kusuoka \cite{Ku} can essentially be adapted to the fractional Brownian motion case after suitable changes. Let us however observe that we obtain the correct order in $t$ by using a rescaling 
argument on the vector fields $V_i$'s instead of analyzing the small time behavior of the estimates of Section 2.

\section{Stochastic differential driven by fractional Brownian motions}

In this preliminary section, we present the tools about the stochastic  analysis of fractional Browian motion that are needed for the remainder of the paper. 

\subsection{Fractional Brownian motion}

A fractional Brownian motion $B=(B^1,\cdots,B^d)$ is a $d$-dimensional centered Gaussian process, whose covariance is given by
\begin{equation*}
R\left( t,s\right) :=\mathbb{E}\left(  B_s^j \, B_t^j \right)
=\frac{1}{2}\left( s^{2H}+t^{2H}-|t-s|^{2H}\right),
\quad\mbox{for}\quad 
s,t\in[0,1] \mbox{ and } j=1,\ldots,d.
\end{equation*}
In particular it can be shown, by a standard application of Kolmogorov's criterion, that $B$ admits a continuous version
whose paths are $\ga$-H\"older continuous for any $\ga<H$.

Let $\mathcal{E}$ be the space of $\mathbb{R}^d$-valued step
functions on $[0,1]$, and $\mathcal{H}$  the closure of
$\mathcal{E}$ for the scalar product:
\[
\langle (\mathbf{1}_{[0,t_1]} , \cdots ,
\mathbf{1}_{[0,t_d]}),(\mathbf{1}_{[0,s_1]} , \cdots ,
\mathbf{1}_{[0,s_d]}) \rangle_{\mathcal{H}}=\sum_{i=1}^d
R(t_i,s_i).
\]
When $H>\frac{1}{2}$ it can be shown that $\mathbf{L}^{1/H} ([0,1], \mathbb{R}^d)
\subset \mathcal{H}$, and that for $\phi,\psi \in \mathbf{L}^{1/H} ([0,1], \mathbb{R}^d)$, 
we have
\[
\langle \phi , \psi \rangle_{\mathcal{H}}=H(2H-1)\int_0^1 \int_0^1
\mid s-t \mid^{2H-2} \langle \phi (s) , \psi(t)
\rangle_{\mathbb{R}^d} ds dt.
\]
The following interpolation inequality that was proved in \cite{BH}, will be an essential tool in our analysis. For every  $\gamma >H-\frac{1}{2}$, there exists a constant $C$ such that for every continuous function $f \in \mathcal{H}$,
\begin{align}\label{interpolation H big}
\| f \|_\mathcal{H} \ge C\frac{ \| f \|_{\infty}^{3+\frac{1}{\gamma}}}{\| f \|^{2+\frac{1}{\gamma}}_\gamma},
\end{align}
where 
\[
\| f \|_\gamma= \sup_{0 \le s < t \le 1} \frac{\| f(t)-f(s) \|}{|t-s|^\gamma} +\| f \|_{\infty},
\]
is the usual H\"older norm.

When $\frac{1}{4}<H<\frac{1}{2}$ one has 
$$ \mathcal{H}\subset \mathbf{L}^2([0,1])$$
 and the following interpolation inequality classically holds for every $f \in \mathcal{H}$,
\[
\| f \|_\mathcal{H} \ge C \| f \|_{L^2}.
\]
Let us also mention the following inequality that is useful to bound from below the $L^2$ norm by the supremum norm and the H\"older norm
 \[
\|f\|_{\infty}\le 2\max\left\{ \|f\|_{L^{2}}, \|f\|^{\frac{2\gamma}{2\gamma+1}}_{L^{2}}\|f\|^{\frac{1}{2\gamma+1}}_{\gamma}\right\} .
\]
Such inequality was already used in connection with the space $\mathcal{H}$ in \cite{HP}.

\subsection{Malliavin calculus}

Let us remind the basic framework of Malliavin calculus (see \cite{Nu06} for further details).
A  real valued random variable $F$ is then said to be cylindrical if it can be
written, for a given $n\ge 1$, as
\begin{equation*}
F=
f \left(\int_0^{1} \langle h^1_s, dB_s \rangle ,\ldots,\int_0^{1}
\langle h^n_s, dB_s \rangle \right)\;,
\end{equation*}
where $h^i \in \mathcal{H}$ and $f:\mathbb{R}^n \rightarrow
\mathbb{R}$ is a $C^{\infty}$-bounded function. The set of
cylindrical random variables is denoted $\mathcal{S}$. 

The Malliavin derivative is defined as follows: for $F \in \mathcal{S}$, the derivative of $F$ is the $\mathbb{R}^d$ valued
stochastic process $(\mathbf{D}_t F )_{0 \leq t \leq 1}$ given by
\[
\mathbf{D}_t F=\sum_{i=1}^{n} h^i (t) \frac{\partial f}{\partial
x_i} \left(\int_0^{1} \langle h^1_s, dB_s \rangle ,\ldots,\int_0^{1}
\langle h^n_s, dB_s \rangle \right).
\]
More generally, we can introduce iterated derivatives. If $F \in
\mathcal{S}$, we set
\[
\mathbf{D}^k_{t_1,\ldots,t_k} F = \mathbf{D}_{t_1}
\ldots\mathbf{D}_{t_k} F.
\]
For any $p \geq 1$, it can be checked that the operator $\mathbf{D}^k$ is closable from
$\mathcal{S}$ into $\mathbf{L}^p(\Omega)$. We will denote by
$\mathbb{D}^{k,p}$ the domain of this closure, that is closure of the class of
cylindrical random variables with respect to the norm
\[
\left\| F\right\| _{k,p}=\left( \mathbb{E}\left( |F|^{p}\right)
+\sum_{j=1}^k \mathbb{E}\left( \left\| \mathbf{D}^j F\right\|
_{\mathcal{H}^{\otimes j}}^{p}\right) \right) ^{\frac{1}{p}},
\]
and
\[
\mathbb{D}^{\infty}=\bigcap_{p \geq 1} \bigcap_{k
\geq 1} \mathbb{D}^{k,p}.
\]

For $p>1$ we can consider the  divergence operator $\delta$  which is defined as the adjoint of $\mathbf{D}$ defined on $\mathbf{L}^p(\Omega)$. It is characterized by the duality formula:
\[
\mathbb{E} ( F \delta u) =\mathbb{E} \left( \langle \mathbf{D} F, u \rangle_\mathcal{H} \right), \quad F \in \mathbb{D}^{1,p}.
\]
It is proved in \cite{Nu06}, Proposition 1.5.7 that $\delta$ is continuous from $\mathbb{D}^{1,p}$ into $\mathbf{L}^p(\Omega)$.

\subsection{Stochastic differential equations driven by fractional Brownian motions}

In this paper,  we will consider the following kind of equation:
\begin{equation}
\label{eq:sde} X^{x}_t =x +
\sum_{i=1}^d \int_0^t V_i (X^{x}_s) dB^i_s,
\end{equation}
where the vector fields $V_1,\ldots,V_d$ are $C^\infty$ bounded vector fields on $\R^n$ and where  $B$ is a fractional Brownian motion with parameter $H \in ( 1/4 , 1)$.

If $H > 1/2$. The equation \eqref{eq:sde} is understood in Young's sense, but if $H \in (1/3, 1/2]$, we need to understand the equation in the sense of rough paths theory (see e.g.  \cite{CQ, FV-bk}). In both cases, the $C^\infty$ boundedness of the vector fields is more than enough to ensure the existence and uniqueness of solutions.

Once equation (\ref{eq:sde}) is solved, the vector $X_t^x$ is a typical example of  random variable which can be differentiated in the sense of  Malliavin.  It is classical that one can  express this Malliavin derivative in terms of the first variation process $J $ of the equation, which is defined by the relation $J_{0\to t}^{ij}=\partial_{x_j}X_t^{x,i}$. Setting $\partial V_{j}$ for the Jacobian of $V_{j}$ seen as a function from $\R^{n}$ to $\R^{n}$,  it is well known that $J$ is the unique solution to the linear equation
\begin{equation}\label{eq:jacobian}
J_{0\to t} = I + 
\sum_{j=1}^d \int_0^t \partial V_j (X^{x}_s) \, J_{0 \to s} \, dB^j_s,
\end{equation}
and that the following results hold true (see \cite{CF} and \cite{NS}  for further details):
\begin{proposition}\label{prop:deriv-sde}
Let $X^x$ be the solution to equation (\ref{eq:sde}). Then
for every $i=1,\ldots,n$ and $t>0$, and $x \in \mathbb{R}^n$, we have $X_t^{x,i} \in
\mathbb{D}^{\infty}$ and
\begin{equation*}
\mathbf{D}^j_s X_t^{x}=J_{s\to t} V_j (X_s) , \quad j=1,\ldots,d, \quad 
0\leq s \leq t,
\end{equation*}
where $\mathbf{D}^j_s X^{x,i}_t $ is the $j$-th component of
$\mathbf{D}_s X^{x,i}_t$, $J_{0 \to t}=\partial_{x} X^x_t$ and $J_{s\to t}=J_{0 \to t} J_{0 \to s}^{-1}$. 
\end{proposition}

\smallskip

We finally mention the recent result \cite{CLL}, which gives a useful estimate for moments of the Jacobian of rough differential equations driven by Gaussian processes.
\begin{proposition}\label{prop:moments-jacobian}
Let $p>1/H$. For any $n \ge 0 $, 
\begin{equation}\label{eq:moments-J-pvar}
\mathbb{E} \left( \| J \|^n_{p-var; [0,1]} \right) < +\infty,
\end{equation}
where $\| \cdot \|_{p-var; [0,1]}$  denotes the $p$-variation norm on the interval $[0,1]$.
\end{proposition}

\section{Basic estimates}

Let us consider vector fields $V_1,\cdots, V_d$ on $\mathbb{R}^n$. Let $\mathcal{A}=\{\emptyset\}\cup\bigcup^{\infty}_{k=1}\{1,2,\cdots,d\}^{k}$ and
$\mathcal{A}_{1}=A\setminus\{\emptyset\}$. We say that $I\in \mathcal{A}$ is a word of length $k$ if $I=(i_1,\cdots,i_k)$
and we write $|I|=k$. If $I=\emptyset$, then we denote $|I|=0$. For any integer $l\ge 1$, we denote by $\mathcal{A}(l)$
the set $\{I\in \mathcal{A}; |I|\le l\}$ and by $\mathcal{A}_{1}(l)$ the set $\{I\in \mathcal{A}_{1}; |I|\le l\}$ .
We also define an operation $\ast$ on $\mathcal{A}$ by
$I\ast J=(i_1,\cdots,i_k,j_1,\cdots,j_l)$ for $I=(i_1,\cdots,i_k)$ and $J=(j_1,\cdots, j_l)$ in $\mathcal{A}$.
We define vector fields $V_{[I]}$ inductively by
\[
V_{[j]}=V_{j}, \quad V_{[I\ast j]}=[V_{[I]}, V_{j}], \quad j=1,\cdots,d
\]

Throughout this paper, we will make the following assumptions on the vector fields.
\begin{assumption}\label{UH condition}

\

\begin{enumerate}
 \item  The $V_{i}'s$ are bounded smooth vector fields on $\mathbb{R}^{n}$ with bounded derivatives at any order.\\
 \item There exists an integer $l\ge 1$ and $\omega^{J}_{I}\in C^{\infty}_{b}(\mathbb{R}^{n},\mathbb{R})$ such that for any $x\in \mathbb{R}^{n}$
\begin{align}\label{bracket_omega_UH}
V_{[I]}(x)=\sum_{J\in \mathcal{A}(l)}\omega^{J}_{I}(x)V_{[J]}(x), \quad I\in \mathcal{A}_{1}
\end{align}
\end{enumerate}
\end{assumption}

The second condition was introduced by S. Kusuoka in \cite{Ku}. It holds for a system of vector fields that satisfy a uniform strong H\"omander's bracket generating condition, but observe that in order that Assumption \ref{UH condition} holds,  it is not even  necessary that the bracket generating condition holds.

\
Let us consider the following rescaled differential equations, which depend on the parameter $\epsilon >0$: 
\begin{align}\label{scaled sde}
X^{\epsilon, x}_{t}=&x+\sum^{d}_{i=1}\int^{t}_{0}V^{\epsilon}_{i}(X^{\epsilon, x}_{s})dB^{i}_{s} \notag \\
		   =&x+\sum^{d}_{i=1}\int^{t}_{0}\epsilon^{H}V_{i}(X^{\epsilon, x}_{s})dB^{i}_{s}.
\end{align}
Clearly, the rescaled vector fields $V^{\epsilon}_{i}$ are defined as $V^{\epsilon}_{i}(x)=\epsilon^{H}V_{i}(x)$. More generally, for any 
$I\in \mathcal{A}_{1}(l)$, we denote $V^{\epsilon}_{[I]}(x)=\epsilon^{|I|H}V_{[I]}(x)$. Note that:
\begin{align*}
 V^{\epsilon}_{[I]}(x)=&\epsilon^{|I|H}V_{[I]}(x)\\
                      =&\sum_{J\in \mathcal{A}_{1}(l)}\epsilon^{|I|H}\omega^{J}_{I}(x)V_{[J]}(x)\\
		       =&\sum_{J\in \mathcal{A}_{1}(l)}\epsilon^{(|I|-|J|)H}\omega^{J}_{I}(x)V^{\epsilon}_{[J]}(x)\\
		       =&\sum_{J\in \mathcal{A}_{1}(l)}\omega^{J,\epsilon}_{I}(x)V^{\epsilon}_{[J]}(x)  
\end{align*}
where $\omega^{J,\epsilon}_{I}(x)=\epsilon^{(|I|-|J|)H}\omega^{J}_{I}(x)$.

It is known that for any $\epsilon \in(0,1]$ and 
any $t > 0$, the map $\Phi^{\epsilon}_{t}(x)=X^{\epsilon, x}_{t}: \mathbb{R}^{n}\rightarrow \mathbb{R}^{n}$ is a flow of $C^{\infty}$ 
diffeomorphism (see \cite{FV-bk}). We denote the Jacobian of $\Phi^{\epsilon}_{t}(x)$ by $J^{\epsilon}_{0\rightarrow t}=\frac{\partial X^{\epsilon,x}_{t}}
{\partial x}$. As we mentioned it earlier,  $J^{\epsilon}_{0\rightarrow t}$ and $(J^{\epsilon}_{0\rightarrow t})^{-1}$ satisfies
the following linear equations:

\begin{align}\label{jacobian linear}
dJ^{\epsilon}_{0\rightarrow t}=\sum^{d}_{i=1}\partial V^{\epsilon}_{i}(X^{\epsilon,x}_{t})J^{\epsilon}_{0 \rightarrow t}dB^{i}_{t}, \quad with~ J^{\epsilon}_{0}=I
\end{align}
and
\begin{align}\label{Jacobian inverse linear}
d(J^{\epsilon}_{0 \rightarrow t})^{-1}=-\sum^{d}_{i=1}(J^{\epsilon}_{0 \rightarrow t})^{-1}\partial V^{\epsilon}_{i}(X^{\epsilon, x}_{t})
dB^{i}_{t}, \quad with~(J^{\epsilon}_{0})^{-1}=I
\end{align}
Let us introduce a linear system $\beta^{J,\epsilon}_{I}(t,x)$ that satisfies the following linear equations:
\begin{align}\label{beta}
\begin{cases}
   d\beta^{J,\epsilon}_{I}(t,x)=\displaystyle \sum^{d}_{j=1}\left(\sum_{K\in \mathcal{A}_{1}(l)}-\omega^{K,\epsilon}_{I\ast j}(X^{x,\epsilon}_{t})
   \beta^{J,\epsilon}_{K}(t,x)\right)dB^{j}_{t}\\
   \beta^{J,\epsilon}_{I}(0,x)=\delta^{J}_{I}
\end{cases}
 \end{align}
 
\begin{lemma}\label{jacobian_inverse_beta}
Fix $\epsilon \in (0,1]$. For any $ I\in \mathcal{A}_{1}(l)$, we have:
\[
 (J^{\epsilon}_{0 \rightarrow t})^{-1}(V^{\epsilon}_{[I]}(X^{\epsilon, x}_{t}))=\sum_{J\in \mathcal{A}_{1}(l)}\beta^{J,\epsilon}_{I}(t,x)V^{\epsilon}_{[J]}(x)
\]
\end{lemma}

\begin{proof}
To simpify the notation, let us denote 
\[
a^{\epsilon}_{I}(t,x)= (J^{\epsilon}_{0 \rightarrow t})^{-1}(V^{\epsilon}_{[I]}(X^{\epsilon, x}_{t}))
\]
and 
\[
b^{\epsilon}_{I}(t,x)=\sum_{J\in \mathcal{A}_{1}(l)}\beta^{J,\epsilon}_{I}(t,x)V^{\epsilon}_{[J]}(x)
\]
Clearly by definition, we have $a^{\epsilon}_{I}(0,x)=b^{\epsilon}_{I}(0,x)=V^{\epsilon}_{[I]}(x)$. Next, we show that $a^{\epsilon}_{I}(t,x)$ and
$b^{\epsilon}_{I}(t,x)$ satisfy the same differential equation. Indeed, by change of variable formula, we have:
\begin{align*}
da^{\epsilon}_{I}(t,x)=&d(J^{\epsilon}_{0\rightarrow t})^{-1}(V^{\epsilon}_{[I]}(X^{\epsilon, x}))\\
		      =&\sum^{d}_{j=1}(-1)(J^{\epsilon}_{0\rightarrow t})^{-1}[V^{\epsilon}_{[I]}, V^{\epsilon}_{j}](X^{\epsilon, x}_{t})(x)dB^{j}_{t}\\
		      =&\sum^{d}_{j=1}\sum_{J\in \mathcal{A}_{1}(l)}-\omega^{J,\epsilon}_{I\ast j}(X^{\epsilon, x}_{t})
		      (J^{\epsilon}_{0\rightarrow t})^{-1}V^{\epsilon}_{[J]}(X^{\epsilon,x}_{t})dB^{j}_{t}\\
		      =&\sum^{d}_{j=1}\sum_{J\in \mathcal{A}_{1}(l)}-\omega^{J,\epsilon}_{I\ast j}(X^{\epsilon, x}_{t})a^{\epsilon}_{J}(t,x)dB^{j}_{t}
\end{align*}
on the other hand, by the definition of $\beta^{J, \epsilon}_{I}(t,x)$, we have:
\begin{align*}
db^{\epsilon}_{I}(t,x)=&d(\sum_{K\in \mathcal{A}_{1}(l)}\beta^{K,\epsilon}_{I}(t,x)V^{\epsilon}_{[K]}(x))\\
		       =&\sum_{K\in \mathcal{A}_{1}(l)}d\beta^{K,\epsilon}_{I}(t,x)V^{\epsilon}_{[K]}(x)\\
		       =&\sum_{j=1}^{d}\sum_{J\in \mathcal{A}_{1}(l)}-\omega^{J,\epsilon}_{I\ast j}(X^{\epsilon, x}_{t})
		       \sum_{K\in \mathcal{A}_{1}(l)}\beta^{K,\epsilon}_{J}(t,x)V^{\epsilon}_{[K]}(x)dB^{j}_{t}\\
		       &\sum_{j=1}^{d}\sum_{J\in \mathcal{A}_{1}(l)}-\omega^{J,\epsilon}_{I\ast j}(X^{\epsilon, x}_{t})
		       b^{\epsilon}_{J}(t,x)dB^{j}_{t}
\end{align*}
And the result follows by the uniqueness of solutions.
\end{proof}

The following lemma gives the order of $\beta^{J,\epsilon}_{I}(t,x)$.
\begin{lemma}\label{Taylor expansion of beta}
Let $I,J\in \mathcal{A}_{1}(l)$ such that $|I|\le |J|$, then 
\[
\beta^{J,\epsilon}_{I}(t,x)=\sum_{L\in \mathcal{A}}\delta^{J}_{I\ast L}(-1)^{|L|}B^{K}_{t}+\gamma^{\epsilon,J}_{I}(t,x)
\]
where
\[
\sup_{x \in \mathbb{R}^n} \mathbb{E}\left[\left(\sup_{t\in(0,1], \epsilon \in(0,1]}t^{-(l+1-|I|)H}|\gamma^{\epsilon,J}_{I}(t,x)|\right)^{p}\right]< \infty
\]
holds for any $p \ge 1$.
\end{lemma}
\begin{proof}
Let us consider the  Taylor expansion obtained by iterating the equation (\ref{beta}). Note that since 
\[
V_{[I]}(x)=\sum_{J\in \mathcal{A}_{1}(l)}\omega^{J}_{I}(x)V_{[J]}(x)
\]
then we know that for any $\epsilon \in (0,1]$ and when $|I|\le l$, $\omega^{J,\epsilon}_{I}=\omega^{J}_{I}=\delta^{J}_{I}$. For any $I,J\in \mathcal{A}_{1}(l)$ with $|I|\le|J|$, we have:

\begin{align*}
\beta^{J,\epsilon}_{I}(t,x)=&\delta^{J}_{I}+\displaystyle \sum^{d}_{j=1}\int^{t}_{0}\left(\sum_{K\in \mathcal{A}_{1}(l)}-\omega^{K,\epsilon}_{I\ast j}(X^{\epsilon}_{s})
                             \beta^{J,\epsilon}_{K}(s,x)\right)dB^{j}_{s}\\
                          =&\delta^{J}_{I}+\sum^{d}_{j=1}\int^{t}_{0}(-1)\beta^{J,\epsilon}_{I\ast j}(s,x)dB^{j}_{s}
\end{align*}
Now let us iterate this equation $l-|I|+1$ times and we have:
\begin{align*}
 \beta^{J,\epsilon}_{I}(t,x)=&\delta^{J}_{I}+\sum^{d}_{l_{1}=1}\int^{t}_{0}(-1)\beta^{J,\epsilon}_{I\ast l_{1}}(s_1,x)dB^{l_1}_{s_1}\\
                   =&\delta^{J}_{I}+\sum^{d}_{l_{1}=1}(-1)B^{l_{1}}\delta^{J}_{I\ast l_{1}}+\sum^{d}_{l_1,l_2=1}\int^{t}_{0}
                   \int^{s_{1}}_{0}(-1)^{2}\beta^{J,\epsilon}_{I\ast l_1 \ast l_2}(s_2,x)dB^{l_2}_{s_2}dB^{l_1}_{s_1}\\
                   \cdot \\
                   \cdot \\
                   \cdot \\
                  =&\sum_{L\in \mathcal{A}}\delta^{J}_{I\ast L}(-1)^{|L|}B^{L}_{t}+\sum_{L,j}\sum_{K\in \mathcal{A}_{1}(l)}
                  \int^{t}_{0}\cdots \int^{s_k}_{0}(-1)^{|L|+1}\omega^{K,\epsilon}_{I\ast L\ast j}(X^{\epsilon}_{s_{k+1}})\beta^{J,\epsilon}_{K}(s_{k+1},x)
                  dB^{j}_{s_{k+1}}\cdots dB^{l_{1}}_{s_1}\\
                  =&\sum_{L\in \mathcal{A}}\delta^{J}_{I\ast L}(-1)^{|L|}B^{L}_{t}+\gamma^{\epsilon,J}_{I}(t,x)
\end{align*}
where $\gamma^{\epsilon,J}_{I}(t,x)$ denotes the remainder term. Now,  as an application of  Theorem 10.41 in \cite{FV-bk} (see also \cite{BC}), there exists a random variable $C \in \mathbf{L}^{p}$ such that:
\[
 \|\gamma^{\epsilon,J}_{I}(t,x) \|\le Ct^{(l-|I|+1)H}\sum_{L,j}\sum_{K\in \mathcal{A}_{1}(l)}\|\omega^{K,\epsilon}_{I\ast L\ast j}\|_{Lip^{\gamma-1}}
\]
where $\gamma>1/H$ and $\| \cdot \|_{Lip^{\gamma-1}}$ is the $\gamma-1$-Lipschitz norm.  The result follows then easily.
\end{proof}

\begin{remark}
 Note that 
 \begin{align*}
 \sum_{L\in \mathcal{A}}\delta^{J}_{I\ast L}(-1)^{|L|}B^{L}_{t}=
  \begin{cases}
  (-1)^{|K|}B^{K}_{t}, \quad if J=I\ast K~for~some~K\in \mathcal{A}\\
  0, \quad otherwise
  \end{cases}
  \end{align*}
  Therefore, when $t\rightarrow 0$, the dominating term of $\beta^{\epsilon,J}_{I}(t,x)$ is of order $O(t^{H(|J|-|I|)})$. 
\end{remark}

Now, let us introduce the following notations:
for  any $J \in \mathcal{A}_{1}(l)$,
\[
D^{(J)}f(X^{\epsilon, x}_{t})=\langle \mathbf{D}_{\cdot}f(X^{x,\epsilon}_{t}), \beta^{J,\epsilon}(\cdot,~x)1_{[0,t]}(\cdot)\rangle_{\mathcal{H}} 
\]
where we denote by $\beta^{J,\epsilon}(\cdot,x)$  the column vector $(\beta^{J, \epsilon}_i(\cdot,x))_{i=1,...,n}.$
For any $I$, $J \in \mathcal{A}_{1}(l)$, we define
\[
 M^{\epsilon}_{I,J}(t,x)=\langle \beta^{I,\epsilon}(\cdot,~x)1_{[0,t]}(\cdot), \beta^{J, \epsilon}(\cdot,~x)1_{[0,t]}(\cdot)\rangle_{\mathcal{H}}.
\]
In the following part, we will only consider the case $t=1$ and we write $M^{\epsilon}_{I,J}(x)$ instead of $M^{\epsilon}_{I,J}(1,x)$. 

The following theorem is the main result of this section and the main technical difficulty of our work:

\begin{theorem}\label{M inverse bound}
 For any $p \in (1, \infty)$, 
 \[
 \sup_{\epsilon \in (0,1], x\in \mathbb{R}^{n}}\mathbb{E}\left(\|(M^{\epsilon}_{I,J}(x))_{I,J\in \mathcal{A}_{1}(l)}\|^{-p}\right)< \infty 
 \]
\end{theorem}

The proof of the Theorem \ref{M inverse bound} is splitted in several steps.

\begin{lemma}\label{sup_norm}
For $ m\ge 0$ and $p \ge 1$, there exists a constant $C_{H,d,p}>0$ such that for any small $\epsilon>0$
\[
\sup_{\sum a^{2}_{I}=1}\mathbb{P}\left( \left\|\sum_{I\in \mathcal{A}(m)}a_{I}B^{I}_{t}\right\|_{\infty,[0,1]}< \epsilon\right)\le C_{H,n,p}\epsilon^{p}
\]
\end{lemma}

\begin{proof} 
We first prove the statement when $H>1/2$. Note that when $m=0$, $\mathcal{A}(m)=\{\emptyset\}$ and $\|a_{\emptyset}\|=1$. The statement is true for any $\epsilon <1$.
When $m=1$, $\mathcal{A}(m)=\{\emptyset, 1,2,\cdots,d \}$. Let $f(t)=a_{\emptyset}+\sum^{d}_{i=1}a_{\{i\}}B^{i}_{t}$. We first assume that $a_{\emptyset}=0$, then 
$f(t)=\sum^{d}_{i=1}a_{\{i\}}B^{i}_{t}$ has the same law as one dimensional fractional Brownian motion $B_{t}$. Then by Theorem 4.6 in \cite{LS} we have:
\[
\mathbb{P}(\|f(t)\|_{\infty, [0,1]}<\epsilon)=\mathbb{P}(\|B_{t}\|_{\infty, [0.1]}<\epsilon)\le C_{H,p}\epsilon^{p}
\]

Now if $a_{\emptyset}\neq 0$, since $f(0)=a_{\emptyset}$, we have:
\begin{align*}
\mathbb{P}(\|f(t)\|_{\infty, [0,1]}<\epsilon) \le& \mathbb{P}(\|f(t)\|_{\infty, [0,1]}<\epsilon,~|a_{\emptyset}|\ge \epsilon)
+\mathbb{P}(\|f(t)\|_{\infty, [0,1]}<\epsilon, |a_{\emptyset}|<\epsilon)\\
						 =& \mathbb{P}(\|f(t)\|_{\infty, [0,1]}<\epsilon, ~|a_{\emptyset}|<\epsilon)\\
						\le &\mathbb{P}(\|\sum^{d}_{i=1}a_{\{i\}}B_{t}^{i}\|_{\infty, [0,1]}-|a_{\emptyset}|<\epsilon,~|a_{\emptyset}|<\epsilon)\\
						\le &\mathbb{P}(\|\sum^{d}_{i=1}a_{\{i\}}B_{t}^{i}\|_{\infty, [0,1]}<2\epsilon)\\
						\le &\mathbb{P}\left(\left\|\sum^{d}_{i=1}\frac{a_{\{i\}}}{\sqrt{\sum a^{2}_{\{i\}}}}B^{i}_{t}\right\|_{\infty, [0,1]}<
									              \frac{2\epsilon}{\sqrt{\sum a^{2}_{\{i\}}}}\right)
\end{align*}
Note that when $|a_{\emptyset}|<\epsilon$, we have $\sum^{d}_{i=1}a^{2}_{\{i\}}\ge 1-\epsilon^{2}$. Therefore when $\epsilon<\frac{\sqrt{3}}{2}$, we have 
\begin{align*}
\mathbb{P}(\|f(t)\|_{\infty, [0,1]}<\epsilon)\le& \mathbb{P}\left(\left\|\sum^{d}_{i=1}\frac{a_{\{i\}}}{\sqrt{\sum a^{2}_{\{i\}}}}B^{i}_{t}\right\|_{\infty, [0,1]}< 4\epsilon \right)\\
					      \le& C_{n,p}\epsilon^{p} ,
\end{align*}
where the last inequality follow by the earlier case when $a_{\emptyset}=0$. Now we assume that the statement is true for every $ k=0,1,\cdots, m$. As in the case when $m=1$, 
we may assume that $a_{\emptyset}=0$.Let $\displaystyle f(t)=\sum_{I\in \mathcal{A}_{1}(m+1)}a_{I}B^{I}_{t}$ with the restriction $\displaystyle \sum_{I\in \mathcal{A}_{1}(m+1)}a^{2}_{I}=1$.
Note that $B^{I}_{t}$'s are iterated integrals and we have $\displaystyle B^{I*j}_{t}=\int^{t}_{0}B^{I}_{s}dB^{j}_{s}$. Therefore,
\begin{align*}
f(t)=&\sum_{I \in \mathcal{A}_{1}(m+1)}a_{I}B^{I}_{t}\\
    =&\sum^{d}_{j=1}\int^{t}_{0}\left(\sum_{J \in \mathcal{A}(m)}a_{J*j}B^{J}_{t}\right)dB^{j}_{t},
\end{align*}
where $\displaystyle \sum^{d}_{j=1}\sum_{J\in \mathcal{A}(m)}a^{2}_{J*j}=1$. Now by Propostion 3.4 in \cite{BH}, we have:
\[
\mathbb{P}(\|f(t)\|_{\infty, [0,1]}<\epsilon) \le C_{p}\epsilon^{p}+\min_{j=1,\cdots,n}\left\{ \mathbb{P}\left(\left\|\sum_{J \in \mathcal{A}(m)}a_{J*j}B^{J}_{t}\right\|_{\infty, [0,1]}<\epsilon^{q}\right) \right\}
\]
Note that since $\sum^{d}_{j=1}\sum_{J\in \mathcal{A}(m)}a^{2}_{J*j}=1$, there exists $1\le k \le d$ such that $\sum_{J\in \mathcal{A}(m)}a^{2}_{J*k}\ge \frac{1}{d}$. Therefore,
\begin{align*}
\mathbb{P}(\|f(t)\|_{\infty, [0,1]}<\epsilon)\le& C_{p}\epsilon^{p}+\mathbb{P}\left(\left\|\sum_{J \in \mathcal{A}(m)}a_{J*k}B^{J}_{t}\right\|_{\infty, [0,1]}< \epsilon^{q}\right)\\
                                             \le& C_{p}\epsilon^{p}+\mathbb{P}\left(\left\|\sum_{J\in \mathcal{A}(m)}\frac{a_{J*k}}{\sqrt{\sum a^{2}_{J*k}}}B^{J}_{t}\right\|_{\infty,[0,1]}
                                             < \frac{\epsilon^{q}}{\sqrt{\sum a^{2}_{J*k}}}\right)\\
                                              \le& C_{p}\epsilon^{p}+\mathbb{P}\left(\left\|\sum_{J\in \mathcal{A}(m)}\frac{a_{J*k}}{\sqrt{\sum a^{2}_{J*k}}}B^{J}_{t}\right\|_{\infty,[0,1]}
                                             < \sqrt{d}\epsilon^{q}\right)\\
                                             \le& C_{H,d,p}\epsilon^{p}
\end{align*}  
where the last inequality follows by the induction hypothesis. When $a_{\emptyset}\neq 0$, we repeat the argument in case $m=1$.\\
Now we turn to the irregular case when $1/4\le H\le 1/2$. For the base case $m=0$ or $m=1$, the same argument as in the regular case $H>1/2$ works. We 
just need the irregular version of the Norris lemma ( see Theorem 5.6 in \cite{CHLT}) to run the induction. Assume that the statement is true for 
$k=0,1,\cdots m$. Let $\displaystyle f(t)=\sum_{I\in \mathcal{A}_{1}(m+1)}a_{I}B^{I}_{t}$ with the restriction $\displaystyle \sum_{I\in \mathcal{A}_{1}(m+1)}a^{2}_{I}=1$.\\ 
We have:
\[
f(t)=\int^{t}_{0}A_sdB_{s} ,
\]
where $B_{t}=(B^{1}_{t}, \cdots, B^{d}_{t})$ and $A_{t}=(\sum_{J\in \mathcal{A}(m)}a_{J*1}B^{J}_{t}, \cdots,\sum_{J\in \mathcal{A}(m)}a_{J*d}B^{J}_{t})$. We pick $1 \le k \le d$\
such that $\sum_{J\in \mathcal{A}(m)}a^{2}_{J*k}\ge \frac{1}{d}$. Then by Theorem 5.6 in \cite{CHLT}, we have:
\[
\left\|\sum_{J\in \mathcal{A}(m)}a_{J*k}B^{J}_{t}\right\|_{\infty, [0,1]}\le MR^{q}\|f(t)\|^{r}_{\infty, [0,1]}
\]
Therefore we have:
\begin{align*}
\mathbb{P}(\|f\|_{\infty, [0,1]}< \epsilon)=&\mathbb{P}(\|f\|^{r}_{\infty, [0,1]}< \epsilon^{r})\\
					  \le&\mathbb{P}\left(\frac{\|\sum_{J\in \mathcal{A}(m)}a_{J*k}B^{J}_{t}\|_{\infty, [0,1]}}{MR^{q}}\le \epsilon^{r}\right)\\
					  \le&\mathbb{P}\left(\|\sum_{J\in \mathcal{A}(m)}a_{J*k}B^{J}_{t}\|_{\infty, [0,1]}\le \epsilon^{r/2}\right)+\mathbb{P}\left(MR^{q}\ge \epsilon^{-r/2}\right)\\
					  \le&\mathbb{P}\left(\left\|\sum_{J\in \mathcal{A}(m)}\frac{a_{J*k}}{\sqrt{\sum a^{2}_{J*k}}}B^{J}_{t}\right\|_{\infty, [0,1]}\le \sqrt{d}\epsilon^{r/2}\right)+ C_{p}\epsilon^{p}\\
					  \le&C_{H,d,p}\epsilon^{p}
\end{align*}
The last inequality follows from the induction hypothesis and the fact that $R$ has finite moment of all orders.
\end{proof}

\begin{corollary}\label{L2 norm}
For any $m\ge 0$ and $p>1$, we have 
\[
 \mathbb{E}\left[\inf\left\{\int^{1}_{0}(\sum_{I\in \mathcal{A}(m)}a_{I}B^{I}_{t})^{2}dt; \sum_{I\in \mathcal{A}(m)}a^{2}_{I}=1\right\}^{-p}\right]=C_{H,d,m,p}<\infty
\]

\end{corollary}

\begin{proof}
By Lemma 2.3.1 in \cite{Nu06}, we only need to show that for any $\epsilon>0$, there exists $C_{p}>0$ such that 
\[
\sup_{\sum_{I\in \mathcal{A}(m)}a^{2}_{I}=1}\mathbb{P}\left(\int^{1}_{0}\left(\sum_{I\in \mathcal{A}(m)}a_{I}B^{I}_{t}\right)^{2}dt<\epsilon\right)\le C_{p}\epsilon^{p}
\]
Let us denote that $f(t)=\sum_{I \in \mathcal(A)(m)}a_{I}B^{I}_{t}$. Then we have:
\[
\mathbb{P}\left(\int^{1}_{0}\left(\sum_{I\in \mathcal{A}(m)}a_{I}B^{I}_{t}\right)^{2}dt<\epsilon\right)=\mathbb{P}(\|f\|^{2}_{L^{2}}<\epsilon)=
\mathbb{P}(\|f\|_{L^{2}}<\sqrt{\epsilon})
\]
By using the  interpolation inequality
\[
\|f\|_{\infty}\le 2\max\{ \|f\|_{L^{2}}, \|f\|^{\frac{2r}{2r+1}}_{L^{2}}\|f\|^{\frac{1}{2r+1}}_{r}\} 
\]
we obtain:
\[
\{\|f\|_{L^{2}}< \sqrt{\epsilon}\}\subseteq \left\{ \frac{\|f\|_{\infty}}{2}<\sqrt{\epsilon}, \|f\|_{L^{2}}>\|f\|_{r}\right\}~\cup~
                                                      \left\{\left(\frac{\|f\|_{\infty}}{2\|f\|^{\frac{1}{2r+1}}_{r}}\right)
                                                      ^{\frac{2r+1}{2r}}<\sqrt{\epsilon},~\|f\|_{L^{2}}<\|f\|_{r}\right\}
\]

therfore we have:
\begin{align*}
\mathbb{P}(\|f\|_{L^{2}, [0,1]}<\sqrt{\epsilon})\le&\mathbb{P}(\|f\|_{\infty, [0,1]}<2\sqrt{\epsilon})+\mathbb{P}(\|f\|^{\frac{2r+1}{2r}}_{\infty, [0,1]}<\epsilon^{1/4})+
                                                        \mathbb{P}((2\|f\|^{\frac{1}{2r+1}}_{r})^{\frac{2r+1}{2r}} > \epsilon^{-1/4})\\
                                                   \le& \mathbb{P}(\|f\|_{\infty, [0,1]}<2\sqrt{\epsilon})+\mathbb{P}(\|f\|_{\infty, [0,1]}<\epsilon^{\frac{1}{4r+1}})+
                                                        \mathbb{P}(\|f\|_{r} >2^{-2r-1}\epsilon^{-r/2})
\end{align*}
Therefore, the result follows by Lemma \ref{sup_norm} and the fact that $\|f\|_{r}$ has finite moments of all orders.
\end{proof}

We can observe that thanks to  Corollary \ref{L2 norm}, we have for and $m\ge 0$, $p>1$ and $T,s>0$,
\[
 \mathbb{E}\left[\inf\left\{\int^{T}_{0}(\sum_{I\in \mathcal{A}(m)}a_{I}B^{I}_{t})^{2}dt; \sum_{I\in \mathcal{A}(m)}T^{2|I|H+1}a^{2}_{I}\ge s\right\}^{-p}\right]=C_{H,d,m,p}s^{-p}
\]

\begin{lemma}
Let $m\ge 0$ and $I\in \mathcal{A}(m)$, if $g^{\epsilon}_{I}: (0,1]^{2}\times \Omega \rightarrow \mathbb{R}$ is a continuous process such that:
\[
A_{p}=\sup_{T\in (0,1], \epsilon \in (0,1]}\mathbb{E}\left[\left(T^{-(m+1)H-1/2}\left(\sum_{I\in \mathcal{A}(m)}\int^{T}_{0}(g^{\epsilon}_{I}(t))^{2}dt\right)^{1/2}\right)^{p}\right]< \infty
\]

then 
\[
\mathbb{P}\left(\inf\{\left(\int^{T}_{0}(\sum_{I\in \mathcal{A}(m)}a_{I}(B^{I}_{t}+g^{\epsilon}_{I}(t)))^{2}dt\right)^{1/2}; 
\sum_{I\in \mathcal{A}(m)}T^{2|I|H+1}a^{2}_{I}=1\}\le z^{-1} \right) \le (4^{p}C_{H,d,m,p}+A_{2p})z^{-pr}
\]

for any $T\in (0,1]$ and $z\ge 1$,  $r=\frac{H}{(m+1/2)H+1/2}$ .
\end{lemma}

\begin{proof}
For any $T\in (0,1]$ and  $y\ge 1$, we have 
\begin{align*}
 &\left(\int^{T}_{0}\left(\sum_{I\in \mathcal{A}(m)}a_{I}(B^{I}_{t}+g^{\epsilon}_{I}(t))\right)^{2}dt\right)^{1/2}\\
 &\ge\left(\int^{T/y}_{0}\left(\sum_{I\in \mathcal{A}(m)}a_{I}(B^{I}_{t}+g^{\epsilon}_{I}(t))\right)^{2}dt\right)^{1/2}\\
 &\ge\left(\int^{T/y}_{0}\left(\sum_{I\in \mathcal{A}(m)}a_{I}B^{I}_{t}\right)^{2}dt\right)^{1/2}-\left(\sum_{I\in \mathcal{A}(m)}T^{2|I|H+1}a^{2}_{I}\right)^{1/2}
                                                                          \left(T^{-(2mH+1)}\sum_{I\in \mathcal{A}(m)}\int^{	T/y}g^{\epsilon}_{I}(t)^{2}dt\right)^{1/2}                                                              
\end{align*}
Now let us pick $z=y^{(m+1/2)H+1/2}$, we have
\begin{align*}
&\mathbb{P}\left(\inf\{\left(\int^{T}_{0}(\sum_{I\in \mathcal{A}(m)}a_{I}(B^{I}_{t}+g^{\epsilon}_{I}(t)))^{2}dt\right)^{1/2}; 
\sum_{I\in \mathcal{A}(m)}T^{2|I|H+1}a^{2}_{I}=1\}\le z^{-1} \right)\\
&\le \mathbb{P}\left(\inf\{\left(\int^{T/y}_{0}(\sum_{I\in \mathcal{A}(m)}a_{I}B^{I}_{t})^{2}dt\right)^{1/2}; 
\sum_{I\in \mathcal{A}(m)}T^{2|I|H+1}a^{2}_{I}=1\}\le 2z^{-1} \right)\\
& +\mathbb{P}\left(T^{-(2mH+1)/2}\left(\sum_{I\in \mathcal{A}(m)}\int^{T/y}(g^{\epsilon}_{I}(t))^{2}dt\right)^{1/2} \ge z^{-1}\right)\\
& \le \mathbb{P}\left(\inf\{\int^{T/y}_{0}(\sum_{I\in \mathcal{A}(m)}a_{I}B^{I}_{t})^{2}dt; 
\sum_{I\in \mathcal{A}(m)}(T/y)^{2|I|H+1}a^{2}_{I}\ge y^{-(2mH+1)}\}\le 4z^{-2} \right)\\
& +\mathbb{P}\left((T/y)^{-(m+1)H-1/2}\left(\sum_{I\in \mathcal{A}(m)}\int^{T/y}(g^{\epsilon}_{I}(t))^{2}dt\right)^{1/2} \ge y^{(m+1)H+1/2}z^{-1}\right)\\
& \le \left(4z^{-2}y^{2mH+1}\right)^{p}C_{m,n,p}+\left(y^{-(m+1)H-1/2}z\right)^{2p}A_{2p}\\
& \le \left(4^{p}C_{H,d,m,p}+A_{2p}\right)y^{-Hp}\\
& \le\left(4^{p}C_{H,d,m,p}+A_{2p}\right)z^{-rp}
\end{align*}
\end{proof}

Now, by  applying the above lemma with $m=l-1$ and Lemma \ref{Taylor expansion of beta}, we obtain the following corollary: 

\begin{corollary}\label{L_2 lower bound}
For any $p \ge 1$ and $\delta>0$, there exists a constant $C_{p}$ such that 
\[
 \mathbb{P}\left(\inf\left\{\sum_{I,J\in \mathcal{A}_{1}(l)}\int^{t}_{0}t^{-(|I|+|J|-2)H+1}a_{I}a_{J}\langle \beta^{I, \epsilon}(s,x), \beta^{J, \epsilon}(s,x)\rangle_{\mathbb{R}^{d}}ds; 
 \sum_{I\in \mathcal{A}_{1}(l)}|a_{I}|^{2}=1\right\}\le \delta\right)\le C_{p}\delta^{p}
\]
for any $\epsilon \in (0, 1]$ and any $x \in \mathbb{R}^{n}$.
\end{corollary}

We are finally in position to give the proof of Theorem \ref{M inverse bound}. First, let us recall that $M^{\epsilon}_{I,J}(x)=\langle \beta^{I,\epsilon}(\cdot, x), \beta^{J,\epsilon}(\cdot, x)\rangle_{\mathcal{H}}$. 
We separate the case $1/4<H \le 1/2$ and $H>1/2$, since we are using different interpolation inequalities for each case. When $1/4<H\le 1/2$, for any $a\in \mathbb{R}^{\mathcal{A}_{1}(l)}$ we have:
\begin{align*}
\sum_{I,J\in \mathcal{A}_{1}(l)}a_{I}a_{J}M^{\epsilon}_{I,J}(x)&=\sum_{j=1}^{d}\|\sum_{I\in \mathcal{A}_{1}(l)}a_{I}\beta_{j}^{I,\epsilon}(\cdot, x)\|^{2}_{\mathcal{H}}\\
                                                     &\ge C_{H}\sum^{d}_{j=1}\int^{1}_{0}(\sum_{I\in \mathcal{A}_{1}(l)}a_{I}\beta_{j}^{I,\epsilon}(t,x))^{2}dt\\
						       &= C_{H}\sum_{I,J\in \mathcal{A}_{1}(l)}\int^{1}_{0}a_{I}a_{J}\langle \beta^{I,\epsilon}(t,x), \beta^{J,\epsilon}(t,x)\rangle_{\mathbb{R}^{d}}dt
\end{align*}                         
Therefore we conclude that: 
\[
 \mathbb{P}\left(\inf\left\{\sum_{I,J\in \mathcal{A}_{1}(l)}a_{I}a_{J}M^{\epsilon}_{I,J}(x); \sum_{I\in \mathcal{A}_{1}(l)}|a_{I}|^{2}=1\right\}\le \delta\right)\le C_{p,H}\delta^{p},
\]
by applying the Corollary \ref{L_2 lower bound} above when $t=1$. Now we turn to the case when $H>1/2$. To simpify the notation, let us denote 
$f_{j}=\sum_{I\in \mathcal{A}_{1}(l)}a_{I}\beta^{I, \epsilon}_{j}(t,x)$. Applying the interpolation inequality (\ref{interpolation H big}) and note that $\|f_j\|_\infty\geq \|f_j\|_{L^2}$ on the interval $[0,1]$, we have:
\begin{align*}
\sum_{I,J\in \mathcal{A}_{1}(l)}a_{I}a_{J}M^{\epsilon}_{I,J}(x)&=\sum_{j=1}^{d}\|\sum_{I\in \mathcal{A}_{1}(l)}a_{I}\beta_{j}^{I,\epsilon}(\cdot, x)\|^{2}_{\mathcal{H}}\\
                                                               &\ge C_{H}\sum_{j=1}^{d}\left(\frac{\|f_{j}\|^{3+1/\gamma}_{L^2}}{\|f_{j}\|^{2+1/\gamma}_{\gamma}}\right)^{2}\\
                                                               &\ge \frac{C_{H}\sum^{d}_{j=1}\|f_{j}\|^{6+2/\gamma}_{L^2}}{\max_{j=1,\cdots,d}\|f_{j}\|^{4+2/\gamma}_{\gamma}}\\
                                                               &\ge \frac{C_{H}d^{-2-1/\gamma}(\sum^{d}_{j=1}\|f_{j}\|^{2}_{L^2})^{3+1/\gamma}}{\max_{j=1,\cdots,d}\|f_{j}\|^{4+2/\gamma}_{\gamma}}\\
                                                               &=\frac{C_{d,H}\left(\sum_{I,J\in \mathcal{A}_{1}(l)}\int^{1}_{0}a_{I}a_{J}\langle \beta^{I,\epsilon}(t,x), \beta^{J,\epsilon}(t,x)
                                                               \rangle_{\mathbb{R}^{d}}dt\right)^{3+1/\gamma}}{\max_{j=1,\cdots,d}\|f_{j}\|^{4+2/\gamma}_{\gamma}}
\end{align*}
Then we have:
\begin{align*}
&\mathbb{P}\left(\inf\left\{\sum_{I,J\in \mathcal{A}_{1}(l)}a_{I}a_{J}M^{\epsilon}_{I,J}(x); \sum_{I\in \mathcal{A}_{1}(l)}|a_{I}|^{2}=1\right\}\le \delta\right)\\
&\le\mathbb{P}\left(\inf\left\{\sum_{I,J\in \mathcal{A}_{1}(l)}\int^{1}_{0}a_{I}a_{J}\langle \beta^{I,\epsilon}(t,x), \beta^{J,\epsilon}(t,x)\rangle_{\mathbb{R}^{d}}dt; 
\sum_{I\in \mathcal{A}_{1}(l)}|a_{I}|^{2}=1\right\}\le \left(\frac{\delta^{1/2}}{C_{d,H}}\right)^{1/(3+1/\gamma)}\right)\\
&+\mathbb{P}\left(\inf\left\{\max_{j=1,\cdots,d}\|f_{j}\|^{4+2/\gamma}_{\gamma};\sum_{I\in \mathcal{A}_{1}(l)}|a_{I}|^{2}=1\right\}\ge \delta^{-1/2}\right)
\end{align*}
and the result follows by chosing $t=1$ in Corollary \ref{L_2 lower bound} and by the fact that $\|f_{j}\|_{\gamma}$ has finite moment of all orders.

\section{Integration by parts formula}

In this section, we will the integration by parts formula which leads to our main result. 
\begin{proposition}
 For any $f \in C^{\infty}_{b}(\mathbb{R}^{n}, \mathbb{R})$, $\epsilon \in (0,1]$ and $x\in \mathbb{R}^{n}$, we have 
 \[
 V^{\epsilon}_{[I]}f(X^{\epsilon, x}_{1})=\sum_{J\in \mathcal{A}_{1}(l)}(M^{\epsilon}_{I,J}(x))^{-1}D^{(J)}f(X^{\epsilon,x}_{1})
 \]

\end{proposition}

\begin{proof}
First note that by chain rule together with Lemma \ref{jacobian_inverse_beta} we have:
\begin{align*}
 \mathbf{D}^{j}_{t}f(X^{\epsilon, x}_{1})=&\langle \nabla{f(X^{\epsilon, x}_{1})}, \mathbf{D}^{j}_{t}X^{\epsilon, x}_{1}\rangle_{\mathbb{R}^{n}}\\
                  =&\langle \nabla{f(X^{\epsilon, x}_{1})}, J^{\epsilon}_{0\rightarrow 1}(J^{\epsilon}_{0\rightarrow t})^{-1}
                  V^{\epsilon}_{j}(X^{\epsilon, x}_{s})\rangle_{\mathbb{R}^{n}}\\
                  =&\langle \nabla{f(X^{\epsilon,x}_{t})}, J^{\epsilon}_{0\rightarrow 1}(\sum_{I\in \mathcal{A}_{1}(l)}
                  \beta^{I, \epsilon}_{j}(t,x)V^{\epsilon}_{[I]}(x))\rangle_{\mathbb{R}^{n}}\\
                  =&\langle \nabla{f(X^{\epsilon,x}_{t})}, \sum_{I\in \mathcal{A}_{1}(l)}\beta^{I, \epsilon}_{j}(t,x)
                  J^{\epsilon}_{0\rightarrow 1}V^{\epsilon}_{[I]}(x)\rangle_{\mathbb{R}^{n}}\\
                  =&\sum_{I\in \mathcal{A}_{1}(l)}\beta^{I,\epsilon}_{j}(t,x)V^{\epsilon}_{[I]}f(X^{\epsilon, x}_t)
\end{align*}
Now for $J\in \mathcal{A}_{1}(l)$, by definition, we have:
\begin{align*}
 D^{(J)}f(X^{\epsilon, x}_1)=&\langle \mathbf{D}_{\cdot}f(X^{\epsilon, x}_{1}), ~\beta^{J, \epsilon}(\cdot,~ x)\rangle_{\mathcal{H}}\\
              =&\langle \sum_{I\in \mathcal{A}_{1}(l)}\beta^{I, \epsilon}(\cdot, x)V^{\epsilon}_{[I]}f(X^{\epsilon, x}_{1}), ~\beta^{J,\epsilon}(\cdot, x)
              \rangle_{\mathcal{H}}\\
              =&\sum_{I\in \mathcal{A}_{1}(l)} V^{\epsilon}_{[I]}f(X^{\epsilon, x}_1)\langle \beta^{I,\epsilon}(\cdot,~x), \beta^{J,\epsilon}(\cdot,~x)\rangle_{\mathcal{H}}\\
              =&\sum_{I\in \mathcal{A}_{1}(l)}M^{\epsilon}_{I,J}(x)V^{\epsilon}_{[I]}f(X^{\epsilon, x}_{1})
\end{align*}
Hence we conclude
\[
 V^{\epsilon}_{[I]}f(X^{\epsilon, x}_{1})=\sum_{J\in \mathcal{A}_{1}(l)}(M^{\epsilon}_{I,J}(x))^{-1}D^{(J)}f(X^{\epsilon, x}_{1})
 \]
\end{proof}

Let us introduce the following definition:
\begin{definition}
We denote by $\mathcal{K}$ the set of mappings 
 $\Phi(\epsilon, x): (0,1]\times \mathbb{R}^{n}\rightarrow \mathbb{D}^{\infty}$ that satisfies the following conditions:
 \begin{enumerate}
  \item $\Phi(\epsilon,x)$ is smooth in x and $\frac{\partial^{|\nu|}\Phi}{\partial x^{\nu}}(\epsilon, x)$ is continues in $(\epsilon,x)\in (0,1]\times \mathbb{R}^{n}$with 
  probability one for any muti-index $\nu$;
  \item For any $k,p>1$ and multi-index $\nu$ we have:
  \[
  \sup_{\epsilon \in(0,1]} \left\|\frac{\partial^{|\nu|}\Phi}{\partial^{\nu}x}(\epsilon,x)\right\|_{\mathbb{D}^{k,p}} < \infty.
  \]
  \end{enumerate}
\end{definition}

\begin{lemma}

\

 \begin{enumerate}
  \item $\beta^{J, \epsilon}_{I}(1,x)\in \mathcal{K}$ for any $I, J\in \mathcal{A}_{1}(l)$.
  \item $(M^{\epsilon}_{I,J}(x))^{-1}\in \mathcal{K}$ for any $I, J\in \mathcal{A}_{1}(l)$.
  \item $\Psi_{I}(\epsilon,t,x)=\sum_{J\in \mathcal{A}_{1}(l)}\beta^{J,\epsilon}(t,x)(M^{\epsilon}_{I,J}(x))^{-1}\in \mathcal{K}$.
 \end{enumerate}
 \end{lemma}
 
 \begin{proof} 
This is a direct consequence of Lemma \ref{Taylor expansion of beta} and Theorem \ref{M inverse bound}.
 \end{proof}

\begin{proposition}\label{by parts}
Let $\Phi(\epsilon,x)\in \mathcal{K}$, then for any $I\in \mathcal{A}_{1}(l)$ , there exists $T^{\ast}_{V^{\epsilon}_{[I]}}\Phi(\epsilon,x)\in \mathcal{K}$ such that 
\[
\mathbb{E}(\Phi(\epsilon,x)V^{\epsilon}_{[I]}f(X_{1}^{\epsilon, x}))=\mathbb{E}\left(f(X^{\epsilon, x}_{1})T_{V^{\epsilon}_{[I]}}^{\ast}\Phi(\epsilon,x)\right).
\]
\end{proposition}

\begin{proof}
We have
\begin{align*}
 \mathbb{E}(\Phi(\epsilon,x)V_{[I]}f(X^{\epsilon, x}_{1}))=&\mathbb{E}\left(\Phi(\epsilon,x)\sum_{J\in \mathcal{A}_{1}(l)}
                                                      (M^{\epsilon}_{I,J}(x))^{-1}D^{(J)}f(X^{\epsilon, x}_{1})\right)\\
                             =&\mathbb{E}\left(\Phi(\epsilon,x)\sum_{J\in \mathcal{A}_{1}(l)}(M^{\epsilon}_{I,J}(x))^{-1}
                             \langle \mathbf{D}_{\cdot}f(X^{\epsilon, x}_{1}), \beta^{J,\epsilon}(\cdot,x)\rangle_{\mathcal{H}}\right)\\
                             =&\mathbb{E}\left(\langle \mathbf{D}_{\cdot}f(X^{\epsilon, x}_{1}), \sum_{J\in \mathcal{A}_{1}(l)}
                             \beta^{J,\epsilon}(\cdot,x)(M^{\epsilon}_{I,J}(x))^{-1}\Phi(\epsilon,x)\rangle_{\mathcal{H}}\right)\\
                             =&\mathbb{E}\left(f(X^{\epsilon, x}_{1})T^{\ast}_{V^{\epsilon}_{[I]}}\Phi(\epsilon,x)\right)
\end{align*}

where
\begin{align*}
T^{\ast}_{V^{\epsilon}_{[I]}}\Phi(\epsilon,x)=&\delta\left( \sum_{J\in \mathcal{A}_{1}(l)}\beta^{J,\epsilon}(t,x)
                                              (M^{\epsilon}_{I,J}(x))^{-1}\Phi(\epsilon,x)\right)\\
                                            =&\delta\left(\Psi_{I}(\epsilon,t,x)\Phi(\epsilon,x)\right).
\end{align*}

Then,  by using  the continuity of $\delta: \mathbb{D}^{k+1}\rightarrow \mathbb{D}^{k}$ and H\"older's inequality we have:
\begin{align*}
\|T^{\ast}_{V^{\epsilon}_{[I]}}\Phi(\epsilon,x)\|_{\mathbb{D}^{k,p}}\le& C_{k,p}\|\Psi_{I}(\epsilon,t,x)\Phi(\epsilon,x)\|_{\mathbb{D}^{k+1,p}}\\
                      \le& C_{k,p}\|\Psi_{I}(\epsilon,t,x)\|_{\mathbb{D}^{k+1,r}}\|\Phi(\epsilon,x)\|_{\mathbb{D}^{k+1,q}}
\end{align*}
where $\frac{1}{r}+\frac{1}{q}=\frac{1}{p}.$
\end{proof}

\section{Regularization bounds}

Now we are ready to state our main theorem. Consider the equation:
\begin{equation}
\label{eq:sde2} X^{x}_t =x +
\sum_{i=1}^d \int_0^t V_i (X^{x}_s) dB^i_s,
\end{equation}
where the vector fields $V_1,\ldots,V_d$ are $C^\infty$ bounded vector fields on $\R^n$ and where  $B$ is a fractional Brownian motion with parameter $H \in ( 1/4 , 1)$.

\begin{theorem}
Let $x\in \mathbb{R}^n$ and $p \ge 1$.  For any integer $k\ge 1$ and $I_{1}, \cdots, I_{k}\in \mathcal{A}_{1}(l)$, there exists a constant $C>0$ (depending on $x$) such that for every $C^\infty$ bounded function $f$,
 \[
  \\|V_{[I_{1}]}\cdots V_{[I_{k}]}P_{t}f(x)\\|\le Ct^{-(|I_1|+\cdots +|I_{k}|)H}(P_{t}f^{p}(x))^{\frac{1}{p}}, \quad t \in (0,1].
 \]
\end{theorem}

\begin{proof}
Let $\epsilon=t$. By the fact that $X^{x}_{\epsilon}$ has the same distribution as $X^{\epsilon, x}_{1}$, we have:
\begin{align*}
V_{[I_1]}\cdots V_{[I_{k}]}P_{t}f(x)=&V_{[I_1]}\cdots V_{[I_{k}]}\mathbb{E}(f(X^{x}_{t}))\\
                                    =&V_{[I_1]}\cdots V_{[I_{k}]}\mathbb{E}(f(X^{x}_{\epsilon}))\\
				     =&\epsilon^{-(|I_{1}|+\cdots |I|_{k})}V^{\epsilon}_{[I_1]}\cdots V^{\epsilon}_{[I_{k}]} \mathbb{E}(f(X^{\epsilon, x}_{1}))
\end{align*}

To prove the theorem, it is sufficient to show that there exists $\Phi(\epsilon, x)\in \mathcal{K}$ such that:
\begin{align}\label{by parts 2}
V^{\epsilon}_{[I_{1}]}\cdots V^{\epsilon}_{[I_{k}]}\mathbb{E}(f(X^{\epsilon, x}_{1}))= \mathbb{E}(f(X^{\epsilon, x}_{1})\Phi(\epsilon,x))
\end{align}
And the result follows by a simple application of H\"{o}lder's inequality. We prove the equation (\ref{by parts 2}) by induction. 
When $k=1$, by Proposition \ref{by parts}, there exists $T^{\ast}_{V^{\epsilon}_{[I_1]}}1(\epsilon,x)\in \mathcal{K}$. Now suppose the statement is true for $k=m$, then there exists 
$\Phi(\epsilon,x)\in \mathcal{K}$ and we have:

\begin{align*}
V^{\epsilon}_{[I_{m+1}]}V^{\epsilon}_{[I_{m}]}\cdots V^{\epsilon}_{[I_{1}]}\mathbb{E}(f(X^{\epsilon, x}_{1}))=&V^{\epsilon}_{[I_{m+1}]}\mathbb{E}(f(X^{\epsilon,x}_{1})\Phi(\epsilon,x))\\
                                                   =&\mathbb{E}\left(\Phi(\epsilon,x)V^{\epsilon}_{[I_{m+1}]}f(X^{\epsilon, x}_{1})
                                                   +f(X^{\epsilon, x}_{1})V^{\epsilon}_{[I_{m}]}\Phi(\epsilon,x)\right)\\
                                                   =&\mathbb{E}\left(f(X^{\epsilon, x}_{1})T^{\ast}_{V^{\epsilon}_{[I_{m+1}]}}
                                                   \Phi(\epsilon, x)+f(X^{\epsilon, x}_{1})V^{\epsilon}_{[I_{m+1}]}\Phi(\epsilon, x)\right)\\
                                                   =&\mathbb{E}\left(f(X^{\epsilon, x}_{1})\left(T^{\ast}_{V^{\epsilon}_{[I_{m+1}]}}\Phi(\epsilon,x)+
                                                   V^{\epsilon}_{[I_{m+1}]}\Phi(\epsilon,x)\right)\right).
\end{align*}
Since by induction hypothesis we know $\Phi(\epsilon, x)\in \mathcal{K}$. Now by Proposition \ref{by parts}, we have that $\left(T^{\ast}_{V^{\epsilon}_{[I_{m+1}]}}\Phi(\epsilon,x)+
                                                   V^{\epsilon}_{[I_{m+1}]}\Phi(\epsilon,x)\right)\in \mathcal{K}$  and this completes the proof
\end{proof}

As a straightforward corollary of the previous result, we finally deduce the following regularization result:

\begin{theorem}
 For any integer $k\ge 1$ and $I_{1}, \cdots, I_{k}\in \mathcal{A}_{1}(l)$, there exists a constant $C>0$ such that for every $C^\infty$ bounded function $f$,
 \[
  \\|V_{[I_{1}]}\cdots V_{[I_{k}]}P_{t}f(x)\\|\le Ct^{-(|I_1|+\cdots+| I_{k}|)H} \| f \|_\infty
 \]
 for any $t\in (0,1]$.
\end{theorem}

\end{document}